\documentclass[11pt, reqno]{amsart}
\usepackage{amsmath, amsthm, amscd, amsfonts, amssymb, graphicx, color, stmaryrd}
\usepackage[all,cmtip]{xy}	
\usepackage[bookmarksnumbered, colorlinks, plainpages]{hyperref}

\newcommand{\Nn}{\mathbb{N}}
\newcommand{\Rr}{\mathbb{R}}
\newcommand{\Cc}{\mathbb{C}}
\newcommand{\Zz}{\mathbb{Z}}

\newcommand{\D}{\mathcal{D}}	

\newcommand{\K}{\mathcal{K}}	
\renewcommand{\L}{\mathcal{L}}	
\newcommand{\A}{\mathcal{A}} 	
\newcommand{\U}{\mathcal{U}}  	
	
\newcommand{\End}{\mathrm{End}}	
\newcommand{\R}{\mathcal{R}}	
\newcommand{\RC}{\mathcal{RC}}	
\newcommand{\tRC}{\widetilde{\RC}} 
\newcommand{\sym}{\mathrm{sym}}
\newcommand{\e}{\underline{e}}	
\newcommand{\m}{\underline{m}}	
\newcommand{\Poles}{\mathcal{P}_{\zeta}} 

\newcommand{\tsigma}{\mathrm{\tilde{\sigma}}}		

\newcommand{\id}{\mathrm{id}}
\newcommand{\Res}{\mathrm{Res}}			
\newcommand{\Wres}{\mathrm{Wres}}		
\newcommand{\Tr}{\mathrm{Tr}}			
\newcommand{\TR}{\mathrm{TR}}
\newcommand{\TRb}{\mathrm{TR}_{\mathrm{bising}}}	
\newcommand{\inv}{\mathrm{inv}}			

\newcommand{\scal}[2]{\langle #1, #2 \rangle}	
\newcommand{\ideal}[1]{\langle #1 \rangle} 	

\newcommand{\op}{\operatorname{op}} 		

\newcommand{\potimes}{\hat{\otimes}} 
\newcommand{\up}{\uparrow}
\newcommand{\down}{\downarrow}
\newcommand{\Id}{\operatorname{Id}}	

\newcommand{\iso}{\xrightarrow{\sim}}	

\newtheorem{Thm}{Theorem}[section]
\newtheorem{Lem}[Thm]{Lemma}
\newtheorem{Prop}[Thm]{Proposition}

\theoremstyle{definition}
\newtheorem{Def}[Thm]{Definition}

\newtheorem{Rem}[Thm]{Remark}
\newtheorem{Ass}[Thm]{Assumption}

\begin{document}
\setcounter{page}{1}


\title{On the $\eta$-function for bisingular pseudodifferential operators}

\author[Karsten Bohlen]{Karsten Bohlen}

\address{$^{1}$ Leibniz University Hannover, Germany ~\textsf{bohlen.karsten@math.uni-hannover.de}}






\begin{abstract}
In this work we consider the $\eta$-invariant for pseudodifferential operators of tensor product type, also called bisingular pseudodifferential operators.
We study complex powers of classical bisingular operators. 
We prove the trace property for the Wodzicki residue of bisingular operators and show how the residues of the $\eta$-function
can be expressed in terms of the Wodzicki trace of a projection operator.
Then we calculate the $K$-theory of the algebra of $0$-order (global) bisingular operators. 
With these preparations we establish the regularity properties of the $\eta$-function at the origin for global bisingular operators which are self-adjoint, elliptic and of positive orders.
\end{abstract} \maketitle



\section{Introduction}


The theory of pseudodifferential operators is indispensible in the study of partial differential equations and index theory as they occur naturally when 
we consider differential operators and in the construction of parametrices of differential operators which are Fredholm.
In the landmark papers of Atiyah and Singer (cf. \cite{as}, \cite{asIV}) the authors consider 
tensor products of complexes of pseudodifferential operators.
Such tensor products are no longer contained in the ordinary H\"ormander's classes of pseudodifferential operators.
Here the calculus of bisingular operators is the correct class which allows for a systematic treatment of tensor products
while basic pseudodifferential techniques are still applicable. 
This class of operators contains tensor products of classical pseudodifferential operators as well as the external product
of such operators. 
The bisingular calculus was introduced in 1975 by L. Rodino, \cite{rodino}. 
In this paper we continue the study of the calculus. 
There are many questions still unresolved, for example it was outside the scope of the original papers of 
Atiyah and Singer to obtain an analogue of their index formula for pseudodifferential operators of tensor product type.
The main difficulty lies in the operator valued nature of the principal symbols and the non-commutativity of the symbol space.
Despite the importance of the index problem it is still unresolved (see also \cite{nr} for recent work).

In the work of Atiyah, Patodi and Singer (cf. \cite{aps2}) on the index formula for manifolds which do possess a boundary, the $\eta$-invariant enters.
The $\eta$-invariant can be defined in terms of a regularized trace of a given pseudodifferential operator (self-adjoint and of positive order).
Roughly speaking and in mild cases the $\eta$-function measures the difference between the number of the positive and the negative eigenvalues of a given self-adjoint operator.
In this connection Atiyah, Patodi and Singer considered the question of the regularity of the $\eta$-function at the origin and
established such a result for particular cases.
The complete solution was obtained later by Gilkey, \cite{gi2}. 


In the sequel we give an outline of the contents of this paper. 
The $\eta$-function for a self-adjoint, elliptic bisingular operator $A$ is defined via the bisingular canonical trace
\begin{align}
\eta(A, z) &:= \TRb A |A|^{-z + 1}, z \in \Cc. \label{eta}
\end{align}

The $\eta$-function is meromorphic with poles of first and possibly second order. 
In the standard case it is an important and non-trivial result that the $\eta$ function is regular and analytic
at the origin $z = 0$. 
For a proof see e.g. Gilkey \cite{gi2}, \cite{gi3} as well as the exposition due to Ponge, \cite{ponge}. 

We reiterate the definition of the spectral $\zeta$-function and introduce the $k$-th Wodzicki residue for bisingular operators.
It was defined in \cite{nr} as follows. For a given bisingular operator $A \in \Psi_{cl}^{m_1,m_2}(X_1 \times X_2)$ we
fix two positive, elliptic operators $Q_1 \in \Psi_{cl}^1(X_1), \ Q_2 \in \Psi_{cl}^1(X_2)$ and set
\begin{align}
\Wres^{(k)}(A) &:= \Res_{z=0}^k \Tr(A Q_1^{-z} \otimes Q_2^{-z}). \label{Wres}
\end{align}

Starting from this definition we show that $\Wres^2$ represents a trace on $\Psi_{cl}^{\Cc, \Cc}(X_1 \times X_2) / \Psi^{-\infty, -\infty}(X_1 \times X_2)$ 
by adapting the argument of Wodzicki, see e.g. \cite{kassel}. 

The proof of the regularity properties of the $\eta$-function relies ultimately on an important relationship between the residue of the $\eta$-function and the 
Wodzicki residue of certain projections.
These projections $\Pi_{+}(A)$ are called \emph{sectorial projections} in the non-selfadjoint case.
The result can be stated (in our context) as follows for an elliptic bisingular pseudodifferential operator $A$ of positive order which is self-adjoint
\begin{align}
\Res_{z = 0}^2 \eta(A, z) &= 2 \pi i \Wres^{2} \Pi_{+}(A) \label{proj} 
\end{align}

see also \cite{w1}. 

Hence this is expressed in terms of the Wodzicki residue of a projection operator in the pseudodifferential calculus.
The formula above holds for the case of self-adjoint bisingular pseudodifferential operators.
For two proofs that the sectorial projections of a classical elliptic pseudodifferential operator of positive order in the standard case is a pseudodifferential operator of order $\leq 0$ we refer to \cite{bclz} and \cite{grubb}. 
Since we are in this note only concerned with the regularity of the $\eta$-invariant for self-adjoint operators we do not investigate the sectorial projection for non-selfadjoint operators in the bisingular class. 
Using our earlier results on the $K$-theory we prove that the Wodzicki residue of any projection operator is $0$. 
This implies then the main result: The $\eta$-function for self-adjoint, positive order, elliptic 
bisingular operators can have at most first order poles at the origin. 




The paper is organized as follows. In Section 2 we recall elements of the theory of bisingular operator classes, the global calculus and the calculus
for smooth compact manifolds. 
In the third section we calculate the $K$-theory for the norm completions of the algebra of global bisingular operators. 
Then in Section 4 we consider complex powers proving that the \emph{classical} bisingular operators remain in the
classical bisingular class if we take their complex powers.
As a preparation for the discussions in the sequel we will introduce in Section 5 the canonical bisingular trace. 
In Section 6 we recall the definition of the bisingular Wodzicki residue and prove the trace property. 
We also define the spectral $\zeta$-function and the $\eta$-function in this context.
In the final section we give a proof of the holomorphicity properties of the $\eta$-function at the origin. 

\subsection*{Acknowledgements}

For helpful discussions I thank Ubertino Battisti, Magnus Goffeng, Elmar Schrohe and Ren\'e Schulz. 
Part of this research was conducted while I was a member of the Graduiertenkolleg GRK 1463 at Leibniz University of Hannover.
I thank the Deutsche Forschungsgemeinschaft (DFG) for their financial support.

\section{Bisingular operators}

\subsection{Global calculus}

In this section we introduce the terminology and notation for the rest of the paper.
The \emph{global bisingular pseudodifferential calculus} was introduced in the paper \cite{bgpr} and we recall the definition here.

We denote by $\Gamma_{cl}^{m_i}(\Rr^{n_i})$ for $i = 1,2$ the usual classical Shubin classes of symbols (cf. \cite{shubin}) and by
$G_{cl}^{m_i}(\Rr^{n_i})$ the pseudodifferential operators (the quantized symbols in $\Gamma_{cl}^{m_i}$). 

Denote by $\scal{x_i}{\xi_i} := (1 + |x_i|^2 + |\xi_i|^2)^{\frac{1}{2}}$ for $x_i, \ \xi_i \in \Rr^{n_i}$. 

\begin{Def}
The class of bisingular symbols $\Gamma^{m_1, m_2}(\Rr^{n_1 + n_2})$ consists of smooth functions
$a \colon \Rr^{2n_1 + 2n_2} \to \Cc$ with the following uniform estimates
\[
|D_{\xi_1}^{\alpha_1} D_{x_1}^{\beta_1} D_{\xi_2}^{\alpha_2} D_{x_2}^{\beta_2} a(x_1, \xi_1, x_2, \xi_2)| \leq C \scal{x_1}{\xi_1}^{m_1 - |\alpha_1| - |\beta_1|} \scal{x_2}{\xi_2}^{m_2 - |\alpha_2| - |\beta_2|}. 
\]

Furthermore, we set
\[
\Gamma^{-\infty, -\infty}(\Rr^{n_1 + n_2}) = \bigcap_{m_1, m_2} \Gamma^{m_1, m_2}(\Rr^{n_1 + n_2}). 
\]

\label{Def:global}
\end{Def}



Given such a symbol $a$ we have two maps
\[
(x_1, \xi_1) \mapsto a_1(x_1, \xi_1) := ((x_2, \xi_2) \mapsto a(x_1, \xi_1, x_2, \xi_2)) 
\]

and
\[
(x_2, \xi_2) \mapsto a_2(x_2, \xi_2) := ((x_1, \xi_1) \mapsto a(x_1, \xi_1, x_2, \xi_2)). 
\]

Hence $a_1 \in \Gamma^{m_2}(\Rr^{n_2}, \Gamma^{m_1}(\Rr^{n_1}))$ 
and $a_2 \in \Gamma^{m_1}(\Rr^{n_1}, \Gamma^{m_2}(\Rr^{n_2}))$. 

The subclass of bisingular classical symbols is denoted by $\Gamma_{cl}^{m_1, m_2}$ and obtained by using in the above definition the classical Shubin classes.
We will later give an alternative definition of classical symbols based on radial compactifications. 


We have two principal symbols
\begin{align}
\sigma_1^{m_1}(A) &= a_1^{(m_1)} \in C^{\infty}(S^{2n_1 - 1}, G_{cl}^{m_2}(\Rr^{n_2})), \label{p1} \\
\sigma_2^{m_2}(A) &= a_2^{(m_2)} \in C^{\infty}(S^{2n_2 - 1}, G_{cl}^{m_1}(\Rr^{n_1})). \label{p2} 
\end{align}

The principal symbols have the following properties for $A \in G_{cl}^{m_1, m_2}(\Rr^{n_1 + n_2}), \ B \in G_{cl}^{p_1, p_2}(\Rr^{n_1 + n_2})$

\begin{align*}
\sigma_i^{m_i + p_i}(A \cdot B) = \sigma_i^{m_i}(A) \cdot \sigma_i^{p_i}(B), \\ 
\sigma_i^{m_i}(A^{\ast}) = \sigma_i^{m_i}(A)^{\ast}, i = 1,2. 
\end{align*}



Fix the notation $\sigma_{\Rr^{n_1}}, \sigma_{\Rr^{n_2}}$ for the principal symbol map of $G_{cl}^{m_1}(\Rr^{n_1})$ and $G_{cl}^{m_2}(\Rr^{n_2})$
respectively. 
Then define in each case the pointwise principal symbol maps
\begin{align*}
&\tilde{\sigma}_{\Rr^{n_1}} \colon C^{\infty}(S^{2n_2 - 1}, G_{cl}^{m_1}(\Rr^{n_1})) \to C^{\infty}(S^{2n_1 - 1} \times S^{2n_2 - 1}), \\
&\tilde{\sigma}_{\Rr^{n_1}}(F)(x_1, \xi_1, x_2, \xi_2) := \sigma_{\Rr^{n_1}}(F(x_2, \xi_2))(x_1, \xi_1), \ F \in C^{\infty}(S^{2n_2-1}, G_{cl}^{m_1}(\Rr^{n_1})), \\
&\tilde{\sigma}_{\Rr^{n_2}} \colon C^{\infty}(S^{2n_1 - 1}, G_{cl}^{m_2}(\Rr^{n_2})) \to C^{\infty}(S^{2n_1 - 1} \times S^{2n_2 - 1}), \\
&\tilde{\sigma}_{\Rr^{n_1}}(G) := \sigma_{\Rr^{n_2}}(G(x_1, \xi_1))(x_2, \xi_2), \ G \in C^{\infty}(S^{2n_1 - 1}, G_{cl}^{m_2}(\Rr^{n_2}).  
\end{align*}

Note that by nuclearity we have 
\begin{align*}
&C^{\infty}(S^{2n_1 - 1}, G_{cl}^{m_2}(\Rr^{n_2})) \cong C^{\infty}(S^{2n_1 - 1}) \potimes G_{cl}^{m_2}(\Rr^{n_2}), \\
&C^{\infty}(S^{2n_2 - 1}, G_{cl}^{m_1}(\Rr^{n_1})) \cong C^{\infty}(S^{2n_2 - 1}) \potimes G_{cl}^{m_1}(\Rr^{n_1}) 
\end{align*}

and the pointwise symbol maps are also given by
\[
\tilde{\sigma}_{\Rr^{n_1}} = \id_{C^{\infty}(S^{2n_2 -1})} \otimes \sigma_{\Rr^{n_1}}, \ \tilde{\sigma}_{\Rr^{n_2}} = \sigma_{\Rr^{n_2}} \otimes \id_{C^{\infty}(S^{2n_1 - 1})}. 
\]

The following \emph{compatibility condition} holds
\begin{align}
& \sigma_{\Rr^{n_2}}(\sigma_1^{m_1}(A)(x_1, \xi_1))(x_2, \xi_2) = \sigma_{\Rr^{n_1}}(\sigma_2^{m_2}(A)(x_2, \xi_2))(x_1, \xi_1) \notag \\
&= \sigma^{m_1, m_2}(A)(x_1, \xi_1, x_2, \xi_2) = a_{m_1, m_2}(x_1, \xi_1, x_2, \xi_2). \label{comp}
\end{align}

\begin{Def}
Let $\Sigma^{m_1, m_2}$ be the set of all pairs 
\[
(F, G) \in C^{\infty}(S^{2n_1-1}, G_{cl}^{m_2}(\Rr^{n_2})) \oplus C^{\infty}(S^{2n_2-1}, G_{cl}^{m_1}(\Rr^{n_1})) 
\]
such that 
\begin{align*}
\tsigma_{\Rr^{n_2}}(F) &= \tsigma_{\Rr^{n_1}}(G).
\end{align*}

Let $(F_1, G_1) \in \Sigma^{m_1, m_2}, \ (F_2, G_2) \in \Sigma^{p_1, p_2}$ and set
\begin{align*}
(F_2, G_2) \circ (F_1, G_1) &:= (F_2 \circ_{2} F_1, G_2 \circ_{1} G_1) \in \Sigma^{m_1 + p_1, m_2 + p_2}.
\end{align*}

Here
\begin{align*}
(F_2 \circ_{2} F_1)(x_1, \xi_1) &:= F_2(x_1, \xi_1) \circ_{\Rr^{n_2}} F_1(x_1, \xi_1), \\
(G_2 \circ_1 G_1)(x_2, \xi_2) &:= G_2(x_2, \xi_2) \circ_{\Rr^{n_1}} G_1(x_2, \xi_2)
\end{align*}

where $\circ_{\Rr^{n_2}}$ denotes the operator product $G_{cl}^{m_2}(\Rr^{n_2}) \times G_{cl}^{p_2}(\Rr^{n_2}) \to G_{cl}^{m_2 + p_2}(\Rr^{n_2})$ and for $\circ_{\Rr^{n_1}}$ analogously.
\label{Def:symbspace}
\end{Def}

We introduce the appropriate Sobolev spaces for bisingular operators.
\begin{Def}
We define the Sobolev space as the completion 
\[
Q^{s,t}(\Rr^{n_1 + n_2}) = \overline{S(\Rr^{n_1 + n_2})}^{\|\cdot\|_{s,t}}, \ s, t \in \Rr
\]

where the norm is given by
\[
\|u\|_{s,t} := \|\Lambda^{s,t} u\|_{L^2(\Rr^{n_1 + n_2})}. 
\]

Here $\Lambda^{s,t} := \Lambda_{n_1}^{s} \otimes \Lambda_{n_2}^{t}$ and $\Lambda_{n_1}^{s}, \ \Lambda_{n_2}^{t}$ are invertible operators
in the Shubin classes $G_{cl}^{s}(\Rr^{n_1})$ and $G_{cl}^{t}(\Rr^{n_2})$ respectively. 
\label{Def:Sobolev}
\end{Def}

\begin{Prop}[cf. \cite{bgpr}]
Let $P \in G_{cl}^{m_1, m_2}(\Rr^{n_1 + n_2})$ then $P$ has a continuous linear extension 
\[
P \colon Q^{s, t}(\Rr^{n_1+n_2}) \to Q^{s-m_1, t-m_2}(\Rr^{n_1 + n_2}).
\]
\label{Prop:cont}
\end{Prop}

\begin{Rem}
It is not hard to show that we have the isomorphism
\[
Q^{s}(\Rr^{n_1}) \potimes Q^{t}(\Rr^{n_2}) \cong Q^{s, t}(\Rr^{n_1 + n_2})
\]

for the Sobolev spaces on $\Rr^{n_i}, \ i = 1,2$ where by $\potimes$ we denote the completed projective tensor product. 
\label{Rem:Sobolev}
\end{Rem}

\subsection{Smooth, compact manifolds}

On smooth compact manifolds we define the calculus of bisingular pseudodifferential operators. 
We refer the reader to the paper \cite{rodino} for the details.

% ...
\begin{Def}
Let $\Omega_i \subset \Rr^{n_i}$ be open sets for $i = 1,2$. 
Then $a \in S^{m_1, m_2}(\Omega_1 \times \Omega_2)$ if $a \in C^{\infty}(T^{\ast} \Omega_1 \times T^{\ast} \Omega_2)$
such that we have the uniform estimates
\[
|D_{x_1}^{\beta_1} D_{\xi_1}^{\alpha_1} D_{x_2}^{\beta_2} D_{\xi_2}^{\alpha_2} a(x_1, \xi_1, x_2, \xi_2)| \leq C \ideal{\xi_1}^{m_1 - |\alpha_1|} \ideal{\xi_2}^{m_2 - |\alpha_2|}. 
\]

\label{Def:symbols}
\end{Def}

A linear operator $A \colon C_c^{\infty}(\Omega_1 \times \Omega_2) \to C_c^{\infty}(\Omega_1 \times \Omega_2)$ is a \emph{bisingular operator}
if 
\[
(Au)(x_1, x_2) = \frac{1}{(2\pi)^{n_1 + n_2}} \int_{\Rr^{n_1}} \int_{\Rr^{n_2}} e^{ix_1 \xi_ + i x_2 \xi_2} a(x_1, \xi_1, x_2, \xi_2) \hat{u}(\xi_1, \xi_2) \,d\xi_1 \,d\xi_2
\]

for a symbol $a \in S^{m_1, m_2}(\Omega_1 \times \Omega_2)$. 

The subclass $S_{cl}^{m_1, m_2}(\Omega_1 \times \Omega_2)$ denotes the \emph{classical symbols} having a bihomogenous expansion.
Then $\Psi_{cl}^{m_1, m_2}(\Omega_1 \times \Omega_2)$ denote the classical bisingular operators. 


In analogy to the last section we have the symbols defined for $A \in \Psi_{cl}^{m_1, m_2}(\Omega_1 \times \Omega_2)$ as follows
\begin{align*}
&\sigma_1^{m_1}(A) \colon T^{\ast} \Omega_1 \setminus \{0\} \to \Psi_{cl}^{m_2}(\Omega_2), \\
&(x_1, \xi_1) \mapsto a_{m_1, \cdot}(x_1, \xi_1, x_2, D_2), \\
&\sigma_2^{m_2}(A) \colon T^{\ast} \Omega_2 \setminus \{0\} \to \Psi_{cl}^{m_1}(\Omega_1), \\
&(x_2, \xi_2) \mapsto a_{\cdot, m_2}(x_1, D_1, x_2, \xi_2), \\
&\sigma^{m_1, m_2}(A) \colon T^{\ast} \Omega_1 \setminus \{0\} \times T^{\ast} \Omega_2 \setminus \{0\} \to \Cc, \\
&(x_1, \xi_1, x_2, \xi_2) \mapsto a_{m_1, m_2}(x_1, \xi_1, x_2, \xi_2). 
\end{align*}

Additionally, the same compatibility condition \eqref{comp} holds. 

We recall a definition of classical bisingular operators in terms of radial compactifications due to Nicola, Rodino \cite{nr}. 

For this embed $\Rr^{n_i}$ into $S_{+}^{n_i} = \{(\xi', \xi_{n_i + 1}) \in \Rr^{n_i + 1} | |\xi| = 1, \ \xi_{n_i + 1} > 0\}$ via
the homeomorphism given by
\[
\RC_i(\xi) = \left(\frac{\xi}{\ideal{\xi}}, \frac{1}{\ideal{\xi}}\right).
\]

The inverse is given by $\RC^{-1}(z_0, z) = \frac{z_0}{z}$. 

We also define the maps $\tRC_i = \Id \times \RC_i$ which act on the cotangent space $T^{\ast} \Omega_i = \Omega_i \times \Rr^{n_i}$
and map to $S_{+}^{\ast} \Omega_i = \Omega_i \times S_{+}^{n_i}$. 

Fix the projection mapping 
\[
\pi \colon S_{+}^{\ast} \Omega_1 \times S_{+}^{\ast} \Omega_2 \to S_{+}^{n_1 + n_2}.
\]

We therefore consider two manifolds with corners which each consist of two hypersurfaces. For the first case
\[
\begin{cases} S^{n_1 - 1} \times S_{+}^{n_2}, \\
S_{+}^{n_1} \times S^{n_2 - 1} \end{cases}
\]

and we fix the boundary defining functions denoted $\rho_1$ and $\rho_2$ \footnote{i.e. $\rho_i \geq 0$ are smooth such that $\{\rho_i = 0\} = \begin{cases} S^{n_1 - 1} \times S_{+}^{n_2}, & i = 1\\ 
S_{+}^{n_1} \times S^{n_2 - 1}, & i = 2 \end{cases}$ and the $1$-form $d \rho_i$ is non-vanishing on the corresponding boundary hypersurface for $i = 1,2$.}.

For $S_{+}^{\ast} \Omega_1 \times S_{+}^{\ast} \Omega_2$ we have the two boundary hypersurfaces
\[
\begin{cases} \partial(S_{+}^{\ast} \Omega_1) \times S_{+}^{\ast} \Omega_2 = (\Omega_1 \times S^{n_1 - 1}) \times S_{+}^{\ast} \Omega_2, \\
S_{+}^{\ast} \Omega_1 \times \partial(S_{+}^{\ast} \Omega_2) = S_{+}^{\ast} \Omega_1 \times (\Omega_2 \times S^{n_2 - 1})
\end{cases}
\]

and we define the boundary defining functions $\tilde{\rho}_i := \pi^{\ast} \rho_i, \ i = 1,2$. 

We obtain a commuting diagram
\[
\xymatrix{
S_{+}^{\ast} \Omega_1 \times \partial(S_{+}^{\ast} \Omega_2) \ar[d]_{\pi} & & \ar[l(1.4)]^-{\tRC_1 \times \Id} T^{\ast} \Omega_1 \times \partial(S_{+}^{\ast} \Omega_2) \ar[d]_{\pi_0} \\
S_{+}^{n_1} \times S^{n_2 - 1} & & \ar[l(1.5)]^-{\RC_1 \times \Id} \Rr^{n_1} \times S^{n_2 - 1}. 
}
\]

The boundary defining function $\rho_1$ is written
\[
((\RC_1 \times \Id)^{\ast} \rho_1)(\xi_1, \omega_2) = |\xi_1|^{-1}, \ |\xi_1| \geq 1, \ \omega_2 \in S_{+}^{n_2}.
\]

We obtain the commuting diagram
\[
\xymatrix{
\partial (S_{+}^{\ast} \Omega_1) \times S_{+}^{\ast} \Omega_2 \ar[d]_{\pi} & & \ar[l(1.4)]^-{\Id \times \tRC_2} \partial (S_{+}^{\ast} \Omega_1) \times T^{\ast} \Omega_2 \ar[d]_{\pi_0} \\
S^{n_1-1} \times S_{+}^{n_2} & & \ar[l(1.5)]^-{\Id \times \RC_2} S^{n_1-1} \times \Rr^{n_2}. 
}
\]

The boundary defining function $\rho_2$ is written
\[
((\Id \times \RC_2)^{\ast} \rho_2)(\omega_1, \xi_2) = |\xi_2|^{-1}, \ |\xi_2| \geq 1, \ \omega_1 \in S_{+}^{n_1}. 
\]

On the manifold with corners $S_{+}^{\ast} \Omega \times S_{+}^{\ast} \Omega_2$ we define the induced smooth structure
of the smooth manifold $S^{\ast} \Omega_1 \times S^{\ast} \Omega_2$. 
In particular we have the actions on functions coming from the radial compactification map and the inclusion $i \colon S_{+}^{\ast} \Omega \times S_{+}^{\ast} \Omega_2 \hookrightarrow S^{\ast} \Omega_1 \times S^{\ast} \Omega_2$
summarized as follows

\[
\xymatrix{
C^{\infty}(S^{\ast} \Omega_1 \times S^{\ast} \Omega_2) \ar[d]_{i^{\ast}} & & \\
C^{\infty}(S_{+}^{\ast} \Omega_1 \times S_{+}^{\ast} \Omega_2) \ar[r(1.4)]^-{(\tRC_1 \times \tRC_2)_{\ast}} & & C^{\infty}(T^{\ast} \Omega_1 \times T^{\ast} \Omega_2). 
}
\]

Given the actions on functions we obtain the commuting diagram
\[
\xymatrix{
S_{+}^{\ast} \Omega_1 \times S_{+}^{\ast} \Omega_2 \ar[d]_{\pi} & & \ar[l(1.4)]^-{\tRC_1 \times \tRC_2} T^{\ast} \Omega_1 \times T^{\ast} \Omega_2 \ar[d]_{\pi_0} \\
S_{+}^{n_1} \times S_{+}^{n_2} & & \ar[l(1.5)]^-{\RC_1 \times \RC_2} \Rr^{n_1} \times \Rr^{n_2}. 
}
\]

This puts us in a position to give another definition of classical bisingular operators.
\begin{Def}
The classical bisingular pseudodifferential symbols space for orders $(m_1, m_2) \in \Rr^2$ is defined as
\[
S_{cl}^{m_1, m_2}(\Omega_1 \times \Omega_2) = (\tRC_1 \times \tRC_2)^{\ast} \tilde{\rho}_1^{-m_1} \tilde{\rho}_2^{-m_2} C^{\infty}(S_{+}^{\ast} \Omega_1 \times S_{+}^{\ast} \Omega_2).
\]
\label{Def:classical}
\end{Def}

The bisingular smoothing terms are identified with the smooth functions which vanish to all orders on the boundary hyperfaces.

\[
\Psi^{-\infty, -\infty}(\Omega_1 \times \Omega_2) \cong C^{\infty}(\Omega_1 \times \Omega_2; S(\Rr^{n_1 + n_2})) \cong \dot{C}^{\infty}(S_{+}^{\ast} \Omega_1 \times S_{+}^{\ast} \Omega_2). 
\]

We can desribe the operator valued principal symbols for a given $A = \op(a)$.
First let $\tilde{a} \in \tilde{\rho}_1^{-m_1} \tilde{\rho}_2^{-m_2} C^{\infty}(S_{+}^{\ast} \Omega_1 \times S_{+}^{\ast} \Omega_2)$
be the corresponding function.
Then we have a Taylor expansion
\[
\tilde{a} = \sum_{j \leq m_1} \tilde{a}_j \tilde{\rho}_1^{-j}
\]

with the coefficients defined on the boundary hypersurface $\{\tilde{\rho}_1 = 0\}$. 

Then consider the function
\[
(\Id \times \RC_2)^{\ast} \tilde{a}_j \colon (\Omega_1 \times S^{n_1 - 1}) \times T^{\ast} \Omega_2 \to \Cc
\]

and extend this to a function
\[
(\Id \times \tRC_2)^{\ast} \tilde{a}_j \colon (\Omega_1 \times S^{n_1 - 1}) \times T^{\ast} \Omega_2 \to \Cc
\]

which is homogenous of degree $j$ with regard to $\xi_1 \in \Rr^{n_1} \setminus \{0\}$. 
The first principal symbol is then given by
\[
\sigma_1^j(A)(x_1, \xi_1) := ((\Id \times \tRC_2)^{\ast} \tilde{a}_j)(x_1, \xi_1, x_2, D_{x_2}) \in \Psi_{cl}^{m_2}(\Omega_2). 
\]

An analogous construction yields a definition of the second principal symbol $\sigma_2^j(A)$. 

The so defined principal symbols are invariant under changes of coordinates as can be shown (\cite{rodino}).
We obtain in this way a calculus $\Psi^{m_1, m_2}(X_1 \times X_2)$ of pseudodifferential operators on two 
closed compact manifolds $X_1, X_2$. 

Fix a Riemmanian metric $g_i$ on $X_i$ for $i =1,2$ and define $\ideal{\xi} = (1 + |\xi|_{g_i}^2)^{\frac{1}{2}}$ for $i = 1,2$. 


\begin{Def}
Let $X_1, X_2$ be two closed, compact manifolds. We introduce the Sobolev space $H^{s_1, s_2}(X_1 \times X_2)$ as the space
\[
H^{s_1, s_2}(X_1 \times X_2) = \{u \in S'(X_1 \times X_2) : \op(\ideal{\xi_1}^{s_1} \ideal{\xi_2}^{s_2} u) \in L^2(X_1 \times X_2) \}.
\]

With the norm
\[
\|u\|_{s_1, s_2} = \|\op(\ideal{\xi_1}^{s_1} \ideal{\xi_2}^{s_2} u\|_{2} 
\]

for $u \in H^{s_1, s_2}(X_1 \times X_2)$. 
\label{Def:Sobolev}
\end{Def}

It is not hard to show that we have the isomorphism
\[
H^{s_1}(X_1) \potimes H^{s_2}(X_2) \cong H^{s_1, s_2}(X_1 \times X_2)
\]

for the Sobolev spaces on the manifolds $X_i$ which is a Hilbert space tensor product. 
In particular the usual continuity properties hold as well as the immediate analog of Prop. \ref{Prop:cont}. 
Additionally, the space $C^{\infty}(X_1 \times X_2) \cong C^{\infty}(X_1) \potimes C^{\infty}(X_2)$ is the projective limit of the scale of Sobolev spaces.


\section{The $K$-theory}
\label{Kthy}

\subsection{Comparison algebras}

In this section we consider the $C^{\ast}$-algebras obtained by completing the order $(0,0)$ bisingular pseudodifferential operators.
We first discuss these algebras for the calculus on smooth, compact manifolds where it is understood that our results
apply immediately also to the global bisingular calculus.
For the algebras obtained from the global calculus we then calculate the $K$-theory explicitly.

\begin{Lem}
\emph{i)} Let $X_1, X_2$ be two smooth compact manifolds. The $\L(L^2)$-completion of $\Psi_{cl}^{-1,-1}(X_1 \times X_2)$ yields 
\[
\overline{\Psi_{cl}^{-1,-1}} \cong \K := \K(L^2(X_1 \times X_2))
\]

the algebra of compact operators on $L^2$. 

\emph{ii)} We have an isomorphism
\[
\K \cong \K_1 \otimes \K_2
\]

where $\K_i := \K(L^2(X_i)), \ i=1,2$ and $\otimes$ denotes any completed $C^{\ast}$-tensor product.

\label{Lem:tensor}
\end{Lem}

\begin{proof}
\emph{i)} It is a standard argument which demonstrates that the relative $\L(L^2(X_i))$-completions of $\Psi_{cl}^{-1}(X_i)$
yield the compact operators $\K_i$ for $i = 1,2$. 
To see this note that $\Psi_{cl}^{-1}(X_i)$ contains the smoothing operators $\Psi^{-\infty}$ which have a Schwartz
kernel which is rapidly decreasing, hence contained in the Schwartz class $\mathcal{S}(X_i \times X_i)$. 
The rapidly decreasing functions are dense in $L^2(X_i \times X_i)$.
Since operators with $L^2$-kernel are Hilbert-Schmidt (HS) we obtain the inclusions
{ \small{
\[
\xymatrix{
\mathrm{HS}_i \ar@{^{(}->}[r] & \K_i \\
\Psi^{-\infty}(X_i) \ar@{^{(}->}[u] \ar@{^{(}->}[r] & \Psi_{cl}^{-1}(X_i) \ar@{^{(}->}[u] 
}
\]
}}

For the inclusion $\Psi_{cl}^{-1}(X_i) \subset \K_i$ we refer to e.g. \cite{shubin}.
The inclusion $\mathrm{HS}_i \subset \K_i$ is dense, hence the inclusion $\Psi_{cl}^{-1}(X_i) \subset \K_i$ is dense with regard to the
$\L(L^2)$-norm.

The bisingular case is treated in exacly the same way, proving that $\Psi^{-1,-1}(X_1 \times X_2) \subset \K$ is dense. 

\emph{ii)} This a folklore result which can be established by an elementary but tedious argument.
Alternatively, it can be viewed as a special case of a tensor product property of groupoid $C^{\ast}$-algebras
(in this case applied to the pair groupoids $X_i \times X_i, \ i = 1,2$), see \cite{ln}. 
Note that this holds for any $C^{\ast}$-tensor product by the nuclearity of the algebra of compact operators. 
\end{proof}

We make use of a technique which was introduced by H. O. Cordes \cite{c} considering so-called comparison algebras of
pseudodifferential operators.

First recall the definition of comparison algebras for the standard H\"ormander calculus. 

Denote by $X_i$ smooth compact manifolds for $i = 1,2$ and let $\D(X_i)$ denote the algebra of differential operators (filtered
by degree) on $X_i$ for $i = 1,2$. 

Fix the Laplace operators $\Delta_i := \Delta_{g_i}$ for a fixed (smooth) Riemannian metric $g_i$ on $X_i$ and set
$\Lambda^{(i)} := (I + \Delta_i)^{-\frac{1}{2}}$ for $i = 1,2$. 

Define the \emph{comparison algebras} for $i =1,2$ as follows
\begin{align*}
\U(X_i) := \ideal{\{L \Lambda^{(i)} : L \in \D(X_i), \ \mathrm{deg}(L) \leq 1\}, \ \K_i}_{C^{\ast}}. 
\end{align*}

We have the following well-known result.
\begin{Thm}[H. O. Cordes]
Let $X_1, X_2$ be two smooth compact manifolds. The $C^{\ast}$-completion of the $0$-order classical pseudodifferential operators yield
\[
\overline{\Psi_{cl}^{0}(X_i)} \cong \U(X_i), \ i = 1,2.
\]
\label{Thm:Cordes}
\end{Thm}

Our goal is to define the corresponding comparison algebra for the bisingular operators of order $(0,0)$.

\begin{Def}
On $X_1 \times X_2$ define the following \emph{comparison algebra}
\[
\U(X_1 \times X_2) := \ideal{\mathrm{span}\{L_1 \Lambda^{(1)} \otimes L_2 \Lambda^{(2)} : L_j \in \D(X_j), \ \mathrm{deg}(L_j) \leq 1, \ j = 1,2\}, \ \K}_{C^{\ast}}
\]
\label{Def:comparison}
\end{Def}

This leads to a corresponding result for bisingular operators. 

\begin{Thm}
We have an isomorphism of $C^{\ast}$-algebras
\[
\U(X_1 \times X_2) \cong \overline{\Psi_{cl}^{0,0}(X_1 \times X_2)}^{\L(L^2)}. 
\]
\label{Thm:comparison}
\end{Thm}

\begin{proof}
We have the inclusion
\[
\mathrm{span}\{L_1 \Lambda^{(1)} \otimes L_2 \Lambda^{(2)} : \ L_j \in \D(X_j), \ \mathrm{deg}(L_j) \leq 1\} \subset \Psi_{cl}^{0,0}(X_1 \times X_2)
\]

by definition of the bisingular class. 

Additionally, $\K \subset \Psi^{0,0}(X_1 \times X_2)$ and hence the first direction is clear. 

For the other inclusion consider the combined principal symbol $\sigma_1 \oplus \sigma_2 \colon \Psi_{cl}^{0,0}(X_1 \times X_2) \to \Sigma^{0,0}$
and the induced action on $\overline{\Psi_{cl}^{0,0}}$ and $\U(X_1 \times X_2)$ respectively. 
The actions have the same range $\overline{\Sigma^{0,0}} \cong \Sigma$, the completed symbol algebra (see \cite{bohlen}, Lemma 3.3). 
Hence the assertion reduces to the order $(-1,-1)$-case and this follows from Lemma \ref{Lem:tensor} \emph{ii)}. 
\end{proof}

From the standard exact sequences
\[
\xymatrix{
\K_i \ar@{>->}[r] & \overline{\Psi_{cl}^{0}(X_i)} \ar@{->>}[r]^{\sigma_{X_i}} & C(S^{\ast} X_i)
}
\]

we conclude that the completed algebras $\overline{\Psi_{cl}^{0}(X_i)}$ are nuclear each (since the property of nuclearity is
closed under extensions). 

Therefore it does not matter which $C^{\ast}$-tensor product we are using.
Additionally, it is clear by definition that
\[
\U(X_1) \otimes \U(X_2) \cong \U(X_1 \times X_2)
\]

for any $C^{\ast}$-tensor product $\otimes$. 

\subsection{$K$-theory}

Consider now the classical global bisingular operators $G_{cl}^{0,0}(\Rr^{n_1 + n_2})$. 
The Sobolev continuity in particular yields a continuous inclusion $G_{cl}^{0,0} \hookrightarrow \L(L^2(\Rr^{n_1} \times \Rr^{n_2}), L^2(\Rr^{n_1} \times \Rr^{n_2}))$. 
And similarly for the order $\leq 0$ cases. We take the corresponding $\L(L^2)$-completions and obtain $C^{\ast}$-algebras. 

Introduce the respective $\L(L^2)$ completions
\begin{align*}
A_1 &:= \overline{G_{cl}^0(\Rr^{n_1})}, \ A_2 := \overline{G_{cl}^0(\Rr^{n_2})}, \ \K_1 := \overline{G_{cl}^{-1}(\Rr^{n_1})}, \ \K_2 := \overline{G_{cl}^{-1}(\Rr^{n_2})}, \
\A := \overline{G_{cl}^{0,0}}.
\end{align*}

The results we obtained above carry over to this case with the same arguments.

In particular we have the isomorphisms
\[
\A \cong A_1 \otimes A_2, \ \K \cong \K_1 \otimes \K_2. 
\]

We obtain the following $K$-theory\footnote{The calculations for the symbol algebra $\Sigma$ which are not needed in this paper are lengthier, cf. \cite{bohlen}.}.

\begin{Thm}
We have the following $K$-theory
\[
K_0(\A) \cong \Zz, \ K_1(\A) \cong 0. 
\]
\label{Thm:Kthy}
\end{Thm}

\begin{proof}
Consider the short exact sequences induced by the extended principal symbol maps
\begin{align*}
\xymatrix{
0 \ar[r]^{j} & \K_1 \ar[r] & A_1 \ar[r]^-{\sigma_{\Rr^{n_1}}} & C(S^{2n_1 - 1}) \ar[r] & 0, \\
0 \ar[r]^{j} & \K_2 \ar[r] & A_2 \ar[r]^-{\sigma_{\Rr^{n_2}}} & C(S^{2n_2 - 1}) \ar[r] & 0.
}
\end{align*}

Apply the six-term exact sequence in K-theory to obtain
\[
\xymatrix{
\Zz \ar[r] & K_0(A_i) \ar[r] & K_0(A_i / \K_i) \ar[d] \\
K_1(A_i / \K_i) \ar[u]^{\delta} & \ar[l] K_1(A_i) & \ar[l] 0.
}
\]

Since $A_i$ contains the compacts and an element of index one (see e.g. \cite{hoerm}, Thm. 19.3.1) there is a non-unitary isometry which acts with regard to some fixed
orthonormal basis.
Therefore every finite rank projection in the compacts is stably homotopic to $0$ and the map in $K$-theory induced by the inclusion $j$ is the zero map.
Hence the index mapping $\delta$ is surjective. 
Since the $K$-theory of the odd spheres is $K_0(C(S^{2n_i - 1})) \cong K^0(S^{2n_i - 1}) \cong \Zz$ as well as
$K_1(C(S^{2n_i - 1})) \cong K^1(S^{2n_i - 1}) \cong \Zz$ we obtain
\[
K_0(A_i) \cong K_0(A_i / \K_i) \cong \Zz, \ K_1(A_i) \oplus \Zz \cong K_1(A_i / \K_i) \cong \Zz 
\]

hence 
\[
K_0(A_i) \cong \Zz, \ K_1(A_i) \cong 0. 
\]

Since the $K$-theory groups are in particular torsion free we can apply K\"unneth's theorem in $K$-theory to obtain
\[
K_0(\A) \cong K_0(A_1 \otimes A_2) \cong K_0(A_1) \otimes K_0(A_2) \oplus K_1(A_1) \otimes K_1(A_2) \cong \Zz \otimes \Zz \cong \Zz
\]

and
\[
K_1(\A) \cong K_1(A_1 \otimes A_2) \cong K_0(A_1) \otimes K_1(A_2) \oplus K_1(A_1) \otimes K_0(A_2) \cong 0
\]

as claimed.
\end{proof}

\section{Complex powers}

In this section we recall the necessary terminology and assumptions for the study of complex powers of bisingular operators. 
Let $\Lambda \subset \Cc$ denote a sector in the complex plane. 

\begin{Def}
Let $a \in S_{cl}^{m_1, m_2}(X_1 \times X_2)$ be a symbol. We say that $a$ is \emph{$\Lambda$-elliptic} if there is a constant $R > 0$
such that the following conditions hold.

\emph{i)} For each $|v_1| > R$ and $\lambda \in \Lambda$ we have 
\[
\sigma_1^{m_1}(A)(v_1) - \lambda I_{M_2} \in \inv \Psi_{cl}^{m_2}(X_2).
\]

\emph{ii)} For each $|v_1| >R$ and $\lambda \in \Lambda$ we have 
\[
\sigma_2^{m_2}(A)(v_2) - \lambda I_{M_1} \in \inv \Psi_{cl}^{m_1}(X_1).
\]

\emph{iii)} For all $|v_i| > R$ with $i = 1,2$ and all $\lambda \in \Lambda$ we have 
\[
(\sigma^{m_1, m_2}(A)(v_1, v_2) - \lambda)^{-1} \in S^{-m_1, -m_2}(X_1 \times X_2).
\]

\label{Def:lambda}
\end{Def}

In order to introduce complex powers we need to make an additional assumption (cf. e.g. \cite{bgpr}). 

\begin{Ass}
\emph{i)} $A \in \Psi_{cl}^{m_1, m_2}(X_1 \times X_2)$ is $\Lambda$-elliptic.

\emph{ii)} $\sigma(A) \cap \Lambda = \emptyset$ (i.e. $A$ is in particular invertible). 

\label{Ass:A}
\end{Ass}

Then it is possible to define complex powers for bisingular operators.
\begin{Def}
Let $A$ be a bisingular operator fulfilling the assumption \ref{Ass:A}. Then set
\[
A_z := \frac{i}{2\pi} \int_{\partial \Lambda_{\epsilon}^{+}} \lambda^z (A - \lambda I)^{-1}\,d\lambda, \Re(z) < 0
\]

with $\Lambda_{\epsilon} := \Lambda \cup \{z \in \Cc : |z| \leq \epsilon\}$. 

We set 
\[
A^z := A_{z -k} \circ A^k, \ \Re(z- k) < 0. 
\]
\label{Def:cplex}
\end{Def}

\begin{Rem}
\emph{i)} The integral for $A_{z}$ exists by the standard estimate $\|(A - \lambda I)^{-1}\| = O(|\lambda|^{-1})$, cf. \cite{battisti}, Thm. 2.1.

\emph{ii)} Given $A \in \Psi_{cl}^{m_1, m_2}(X_1 \times X_2)$ which admits complex powers we obtain an operator $A^z \in \Psi_{cl}^{m_1 z, m_2 z}(X_1 \times X_2)$, 
see \cite{battisti}, Thm. 2.2. 

\emph{iii)} We recall the following estimates from \cite{battisti}, Lem. 2.1, see also \cite{mss}. 
Let $a$ be a given symbol which is assumed to be $\Lambda$-elliptic. For all $K_i \subseteq \Omega_i$ compact there is a constant $c_0 > 1$ such that for
\[
\Omega_{\xi_1, \xi_2} := \{z \in \Cc \setminus \Lambda : \frac{1}{c_0} \ideal{\xi_1}^{m_1} \ideal{\xi_2}^{m_2} < |z| < c_0 \ideal{\xi_1}^{m_1} \ideal{\xi_2}^{m_2}\}. 
\]

we obtain that 
\[
\mathrm{spec}(a(x_1, \xi_1, x_2, \xi_2)) = \{\lambda \in \Cc : a(x_1, \xi_1, x_2, \xi_2) - \lambda = 0\} \subseteq \Omega_{\xi_1, \xi_2}, \forall \ x_i \in \Omega_i, \ \xi_i \in \Rr^{n_i}. 
\]

We have the estimates
\[
|(\lambda - a_{m_1, m_2}(x_1, \xi_1, x_2, \xi_2))^{-1}| \leq C (|\lambda| + \ideal{\xi_1}^{m_1} \ideal{\xi_2}^{m_2})^{-1}
\]

as well as
\begin{align*}
& |(a_{m_1, \cdot} - \lambda I_1)^{-1}|  \leq C (|\lambda| + \ideal{\xi_1}^{m_1} \ideal{\xi_2}^{m_2})^{-1}, \\
& |(a_{\cdot, m_2} - \lambda I_2)^{-1}| \leq C(|\lambda| + \ideal{\xi_1}^{m_1} \ideal{\xi_2}^{m_2})^{-1}
\end{align*}

for each $x_i \in K_i, \ \xi_i \in \Rr^{n_i}, \ \lambda \in \Cc \setminus \Omega_{\xi_1, \xi_2}, \ i = 1,2$. 
\label{Rem:Batt}
\end{Rem}


We are going to show that complex powers of the classical bisingular operators are again classical bisingular operators.
The proof given here relies on the radial compactification and this method of proof
works also in the SG-calculus (see \cite{bc}, Thm. 1.9).

\begin{Thm}
Let $A \in \Psi_{cl}^{m_1, m_2}(X_1 \times X_2)$ such that $A$ fufills assumption \ref{Ass:A} and $z \in \Cc$ such that $\Re(z) < 0$.
Then the complex power operator $A^z$ is contained in $\Psi_{cl}^{m_1 z, m_2 z}(X_1 \times X_2)$. 
\label{Thm:complex}
\end{Thm}

\begin{proof}
We fix the notation $\m = (m_1, m_2), \ \e = (1,1)$ as well as $\R^{\m} \colon C^{\infty}(S_{+}^{\ast} \Omega_1 \times S_{+}^{\ast} \Omega_2) \to C^{\infty}(T^{\ast} \Omega_1 \times T^{\ast} \Omega_2)$
given by $\R^{\m} f := (\tRC_1 \times \tRC_2)^{\ast} \tilde{\rho}_1^{-m_1} \tilde{\rho}_2^{-m_2} f$. 
By the Cauchy integral formula we obtain
\[
a^z = \frac{1}{2 \pi i} \int_{\partial^{+} \Omega_{\xi_1, \xi_2}} \lambda^z \sym((A - \lambda I)^{-1})\,d\lambda.
\]

We want to show that $a^z \in S_{cl}^{\m z}$. 
First consider
\[
b_{\m z} = \frac{1}{2\pi i} \int_{\partial^{+} \Omega_{\xi_1, \xi_2}} \lambda^z (a_{\m} - \lambda)^{-1}\,d\lambda = [a_{\m}(x_1, \xi_1, x_2, \xi_2)]^z.
\]

It is our aim to prove that $[a_{\m}(x_1, \xi_1, x_2, \xi_2)]^z \in S_{cl}^{\m z}$, i.e. by definition
$(\R^{\m z})^{-1} b_{\m z} \in C^{\infty}(S_{+}^{\ast} \Omega_1 \times S_{+}^{\ast} \Omega_2)$. 

To this end we calculate
\begin{align*}
(\R^{\m z})^{-1} b_{\m z}(y_1, \eta_1, y_2, \eta_2) &= \frac{1}{2\pi i} \int_{\partial^{+} \Omega_{\RC_1^{-1}(\xi_1), \RC_2^{-1}(\xi_2)}} \frac{\lambda^z \tilde{\rho}_1^{m_1 z} \tilde{\rho}_2^{m_2 z}}{a_{\m}(\RC_1^{-1}(\xi_1), \RC_2^{-1}(\xi_2)) - \lambda} \,d\lambda \\
&= \frac{1}{2\pi i} \int_{\partial^{+} \Omega_{\RC_1^{-1}(\xi_1), \RC_2^{-1}(\xi_2)}} \frac{\mu^z \tilde{\rho}_1^{-m_1} \tilde{\rho}_2^{-m_2}}{a_{\m}(\RC_1^{-1}(\xi_1), \RC_2^{-1}(\xi_2)) - \mu \tilde{\rho}_1^{-m_1} \tilde{\rho}_2^{-m_2}} \,d\mu \\
&= \frac{1}{2\pi i} \int_{\partial^{+} \Omega_{\RC_1^{-1}(\xi_1), \RC_2^{-1}(\xi_2)}} \frac{\mu^z}{(\R^{\m})^{-1} a_{\m}(\underline{y}, \underline{\eta}) - \mu} \,d\mu
\end{align*}

via the substitution $\lambda = \mu \tilde{\rho}_1^{-m_1} \tilde{\rho}_2^{-m_2}$. 
By the previous remarks we have the estimate (in radial coordinates)
\[
|(\R^{\m})^{-1} a_{\m}(\underline{y}, \underline{\eta}) - \mu| \geq C (1 + |\mu|). 
\]

Hence we obtain that $(\R^{\m z})^{-1} b_{\m z} \in C^{\infty}(S_{+}^{\ast} \Omega_1 \times S_{+}^{\ast} \Omega_2)$. 

Using the parametrix construction in the bisingular calculus we write first
\[
a^z = \frac{1}{2\pi i} \int_{\partial^{+} \Lambda_{\epsilon}} \lambda^z \tilde{b}(\lambda)\,d\lambda
\]

which equals by the Cauchy integral formula
\[
= \frac{1}{2\pi i} \int_{\Omega_{\xi_1, \xi_2}} \lambda^z \tilde{b}(\lambda)\,d\lambda. 
\]

Recall that
\[
\tilde{b}(\lambda) = \psi_1 (\sigma^{m_1}(A) - \lambda I_2)^{-1} + \psi_2 (\sigma_2^{m_2}(A) - \lambda I_1)^{-1} + \psi_1 \psi_2 (\sigma^{m_1, m_2}(A) - \lambda)^{-1} + c(\lambda)
\]

where 
\[
\lambda c(\lambda) \in S^{-m_1 - 1, -m_2 - 1}(X_1 \times X_2), \ \forall \ \lambda \in \Lambda. 
\]

In particular using the asymptotic expansion argument in the parametrix construction (as in the proof of Thm. 1.9 in \cite{bc})
\[
\int_{\Omega_{\xi_1, \xi_2}} \lambda^z c(\lambda) \,d\lambda \in S_{cl}^{\m z - \e}. 
\]

The cases
\[
\int_{\Omega_{\xi_1, \xi_2}} \lambda^z (a_{m_1, \cdot} - \lambda I_2)^{-1} \,d\lambda \in S_{cl}^{m_1 z, \cdot}
\]

and 
\[
\int_{\Omega_{\xi_1, \xi_2}} \lambda^z (a_{\cdot, m_2} - \lambda I_1)^{-1} \,d\lambda \in S_{cl}^{\cdot, m_2 z}
\]

are proven using the theory of complex powers of pseudodifferential operators on closed manifolds.

Finally, the case $\int (a_{\m} - \lambda)^{-1} \lambda^z \,d\lambda = b_{\m z} \in S_{cl}^{\m z}$ was already established. 
We have therefore shown that $a^z \in S_{cl}^{\m z}$. 
\end{proof}

\section{The bisingular canonical trace}


The so-called \emph{canonical trace} for classical pseudodifferential operators was introduced by M. Kontsevich and S. Vishik, see \cite{kv}.
In this section we outline this construction for the bisingular pseudodifferential operators, thereby obtaining
a definition of a canonical trace functional for bisingular operators (i.e. the \emph{bisingular canonical trace} $\TRb$).

Let $E_i \to X_i, \ i = 1,2$ be two smooth, hermitian vector bundles and let $A \in \Psi_{cl}^{\alpha, \beta}(X_1 \times X_2, E_1 \boxtimes E_2)$. 
Denote by $K(x_1, x_2, y_1, y_2)$ the distributional Schwartz kernel of $A$ and by $a \in S_{cl}^{\alpha, \beta}(X_1 \times X_2, E_1 \boxtimes E_2)$ the symbol. 

In parallel to the construction in \cite{kv} we consider the difference of the Schwartz kernel $K$ (restricted to diagonals of local charts) and the Fourier transforms of the first $N_i + 1, \ N_i \gg 1$ for $i = 1,2$ bihomogenous terms
$a_{\alpha, \beta}, a_{\alpha - 1, \beta - 1}, \cdots, a_{\alpha - N_1, \beta - N_2}$. 

We can write
\begin{align*}
& K_{-n_1 - n_2 + \alpha + \beta + j + k}(x_1, x_2, x_1 - y_1, x_2 - y_2) \\
&= (2\pi)^{-(n_1 + n_2)} \int_{\Rr_{\xi_1}^{n_1}} \int_{\Rr_{\xi_2}^{n_2}} \exp(i (x_1 - y_1) \xi_1 + i (x_2 - y_2) \xi_2) a_{\alpha - j, \beta -k}(x_1, \xi_1, x_2, \xi_2) \,d\xi_1 \,d\xi_2. 
\end{align*}

Here $a$ is positive homogenous in $y_1 - x_1 \in \Rr^{n_1}, \ y_2 - x_2 \in \Rr^{n_2}$ of orders $(-n_1 + \alpha -j, -n_2 + \beta + k)$ and
$\alpha, \beta \notin \Zz$. 
A distribution in $\D'(\Rr^{n_1 + n_2} \setminus \{0\})$ has a unique prolongation to $\D'(\Rr^{n_1 + n_2})$, therefore
$K_{-n_1 + \alpha + j, n_2 - \beta + k}(x_1, x_2, x_1 - y_1, x_2 - y_2)$ makes sense on $y_1 \not= x_1, \ y_2 \not= x_2$. 

Now we consider the (bihomogenous) difference 
\begin{align}
& K(x_1, x_2, y_1, y_2) - \sum_{j=0}^{N_1} \sum_{k=0}^{N_2} K_{-n_1 + \alpha -j, -n_2 + \beta -k}(x_1, x_2, y_1 - x_1, y_2 - x_2). \label{diff}
\end{align}

For $N_1, \ N_2 > 0$ sufficiently large \eqref{diff} defines a continuous function on $(U \times U) \times (V \times V)$ for
$U \subset X_1, \ V \subset X_2$ open sets.
We define the density $t_{U \times V}(A)$ as the restriction of the difference $\eqref{diff}$ to the diagonals $\Delta_U \times \Delta_V$.
The integral density $t_{U \times V}(A)$ restricted to the diagonals $\Delta_U \times \Delta_V$ takes values in $\End(E_1 \boxtimes E_2)$. 

\begin{Def}
The \emph{bisingular canonical trace} is defined as the integrated density
\[
\TRb(A) = \int_{X_1 \times X_2} \operatorname{tr} t(A).
\]
\label{Def:TRb}
\end{Def}

\begin{Rem}
The definition is independent of local coordinates in $X_1 \times X_2$ by the same argument as in \cite{kv}. 
If the real parts of the orders of $A$ are less than $(-n_1, -n_2)$ we obtain the trace
\[
\Tr(A) = \TRb(A)_{|L^2(X_1 \times X_2, E_1 \boxtimes E_2)}. 
\]
\label{Rem:TRb}
\end{Rem}

We can now state the analogues of Lemma 3.1, Lemma 3.2 and Theorem 3.1 of \cite{kv} (the proofs being analogous as well).
In the main theorem we use the notation $\Wres^2$ for the \emph{bisingular Wodzicki trace} as introduced in \cite{nr}, p. 195,
see also Def. \ref{Def:Wres} on p. \pageref{Def:Wres} of the next section. 

\begin{Lem}
The difference \eqref{diff} is continuous on $(U \times U) \times (V \times V)$ for $N_1, \ N_2$ sufficiently large. 
Hence the restriction to the diagonals of the density $t_{U \times V}(A)$ makes sense.
\label{Lem:diff}
\end{Lem}

\begin{Lem}
The density $t_{U \times V}(A)$ with values in $\End(E_1 \boxtimes E_2)$ is for large $N_1, \ N_2$ independent of
local coordinates in $X_1 \times X_2$ and of local trivializations of $E_1 \boxtimes E_2$. 
\label{Lem:trace}
\end{Lem}

\begin{Thm}
The linear functional
\[
\TRb(A) = \int_{X_1 \times X_2} \operatorname{tr} t(A)
\]

for $A \in \Psi_{cl}^{\alpha_0 + \Zz, \beta_0 + \Zz}(X_1 \times X_2, E_1 \boxtimes E_2)$ and $\alpha_0 \in \Cc \setminus \Zz, \ \beta_0 \in \Cc \setminus \Zz$
has the following properties:

\emph{(1)} $\TRb(A)_{|L^2(X_1 \times X_2, E_1 \boxtimes E_2)} = \Tr(A)$ for $\mathrm{Re} \ \operatorname{ord}(A) < (-n_1, -n_2)$. 

\emph{(2)} $\TRb(A)$ is of \emph{trace-type}, i.e. 
\[
\TR([B, C]) = 0 \ \text{for} \ \operatorname{ord} B + \operatorname{ord} C \in \left(\alpha_0 + \Zz, \beta_0 + \Zz\right). 
\]

\emph{(3)} For any doubly holomorphic family $A(z, w)$ of classical bisingular pseudodifferential operators on $X_1 \times X_2$
and $z \in U \subset \Cc, \ w \in V \subset \Cc$ with $\operatorname{ord} A(z, w) = (z, w)$, the function
$\TRb A(z, w)$ is meromorphic with poles at $z = m_1 \in U \cap \Zz, \ w = m_2 \in \Zz \cap V$.
The residues are
\[
\Res_{z= m_1} \Res_{w = m_2} \TRb(A(z, w)) = -\Wres^2 A(z, w). 
\]
\label{Thm:TRb}
\end{Thm}


\section{The Wodzicki trace}


In the following discussion we define the bisingular Wodzicki residue and recall the trace property.
At first we state a folklore result which is easily adapted to the bisingular calculus.

\begin{Lem}
Let $A \in \Psi^{m_1, m_2}(X_1 \times X_2)$ such that $m_1 < -n_1, m_2 < -n_2$ then $A$ is a trace class operator
over $L^2(X_1 \times X_2)$. 
The trace is given by
\[
\Tr A = \int_{X_1 \times X_2} K_{|\Delta} 
\]

where $\Delta \subset (X_1 \times X_2)^2$ denotes the diagonal and $K$ is the Schwartz kernel of $A$.

\label{Lem:trace}
\end{Lem}



We use the notation $\Tr = \TRb$ for the canonical trace which extends the trace in the previous Proposition. 

\begin{Def}
Let $A \in \Psi_{cl}^{m_1, m_2}(X_1 \times X_2)$ be an operator which fulfills the assumption \ref{Ass:A}. 
Then the spectral $\zeta$-function is given by 
\[
\zeta(A, z) := \int_{X_1 \times X_2} K_{A^z}(x_1, x_2, x_1, x_2) \,dx_1\,dx_2 
\]

where $\Re(z) m_1 < -n_1, \Re(z) m_2 < -n_2$.
Here $K_{A^z}$ denotes the kernel of the complex power operator $A^z$. 
\label{Def:zeta}
\end{Def}

Additionally, we fix the set of simple poles of the spectral $\zeta$-function as follows
\[
\Poles := \left\{z_j^{(1)} := \frac{j - n_1}{m_1}, \ z_k^{(2)} = \frac{k - n_2}{m_2} : j, k \in \Nn_0\right\}. 
\]

Note that $\zeta$ can have poles of order two, namely this can occur if $\frac{n_1}{m_1} = \frac{n_2}{m_2}$, see also \cite{battisti}, Theorem 2.3. 



For the definition of the Wodzicki residue we have to consider doubly parametrized holomorphic families of operators $(A(z, \tau))_{(z, \tau) \in \Cc^2}$ as
in \cite{nr}. 

\begin{Prop}[\cite{nr}, Thm. 3.2]
The double $\zeta$-function 
\[
(z, \tau) \mapsto \Tr(A(z, \tau) Q_1^{-z} \otimes Q_2^{-\tau})
\]

is holomorphic for $\Re(z) > m_1 + n_1, \ \Re(\tau) > m_2 + n_2$. 
Furthermore, it can be extended to a meromorphic function with at most simple poles at $z = n_1 + m_1 + j, \ \tau = m_2 + n_2 - k$ where $j, \ k \in \Nn_0$. 
\label{Prop:nicrod}
\end{Prop}

We fix the notation $\Res^k$ for the $k$-th residue. Let $f$ be a meromorphic function with 
Laurent series expansion $f(z) = \sum_{j \leq -\mu} c_j z^{-j}$, we set
\[
\Res^k \sum_{j \leq -\mu} c_j z^{-j} = \begin{cases} c_k, \ k > 0 \\
0, \ k \leq 0 \end{cases}.
\]

\begin{Def}
The $k$-th \emph{Wodzicki-residue} is a linear functional $\Wres^{(k)} \colon \Psi_{cl}^{\Zz, \Zz}(X_1 \times X_2) \to \Cc$ defined by
\[
\Wres^{(k)}(A) = \Res_{z=0}^{(k)} \Tr(A(z,z) Q_1^{-z} \otimes Q_2^{-z}).
\]

Here $Q_i \in \Psi_{cl}^{1}(X_i)$ are positive elliptic operators, $i = 1,2$.

In the case $k = 2$ we call $\Wres^2$ the \emph{bisingular Wodzicki residue}. 

\label{Def:Wres}
\end{Def}

\begin{Rem}
In \cite{nr} the authors introduce additional functionals $\hat{\Tr}_1$ and $\hat{\Tr}_2$.
Set for $Q \in \Psi_{cl}^{m_i}(\Rr^{n_i})$ elliptic and of order one 
\[
\overline{\Tr}_Q(A) = \lim_{z \to 0} \left(\Tr(A Q^{-z}) - \Res_i \frac{A}{z}\right)
\]

for the regularized value at $z = 0$. 
Here we define
\[
\Res_i(A) := \Res_{z=0} \Tr(A Q^{-z}) 
\]

the \emph{Wodzicki residue}. 
One can prove that the Wodzicki residue does not depend on $Q$ and that it defines a trace on $\Psi_{cl}^{\Zz}(X_i)$

Define
\[
\hat{\Tr}_1(A) := (2\pi)^{-n_1} \int_{S^{\ast} X_1} \overline{\Tr}_{Q_1} \sigma_1^{-n_1}(A) \,d\omega_1
\]

and
\[
\hat{\Tr}_2(A) := (2\pi)^{-n_2} \int_{S^{\ast} X_2} \overline{\Tr}_{Q_2} \sigma_2^{-n_2}(A) \,d\omega_2. 
\]

These functionals do not represent a trace on $\Psi_{cl}^{\Zz, \Zz}(X_1 \times X_2)$. 
Though they do yield a trace if restricted to $\Psi_{cl}^{-\infty, \Zz}$ and $\Psi_{cl}^{\Zz, -\infty}$ respectively and the restrictions are
\begin{align*}
\Tr_1(A) &= (2\pi)^{-n_1} \int_{S^{\ast} X_1} \Tr \sigma_1^{-n_1}(A) \,d\omega_1, \ A \in \Psi^{\Zz, -n_2 - 1}(X_1 \times X_2), \\
\Tr_2(A) &= (2 \pi)^{-n_2} \int_{S^{\ast} X_2} \Tr \sigma_2^{-n_2}(A) \,d\omega_2, \ A \in \Psi^{-n_1 -1, \Zz}(X_1 \times X_2). 
\end{align*}

The functionals $\Tr_1$ and $\Tr_2$ which are the restrictions of $\hat{\Tr}_1, \ \hat{\Tr}_2$ to $\Psi_{cl}^{\Zz, -n_2 - 1}(X_1 \times X_2), \ \Psi_{cl}^{-n_1 -1, \Zz}(X_1 \times X_2)$ respectively
do not depend anymore on the choice of $Q_1$ and $Q_2$ (see \cite{nr}, Remark 3.4). 
\label{Rem:Wres}
\end{Rem}

In \cite{nr}, Thm. 3.3. it was shown that the bisingular Wodzicki residue can be expressed solely in terms of the 
scalar principal symbol in the following sense.

\begin{Thm}[\cite{nr}, Thm. 3.3]
Let $A \in \Psi_{cl}^{m_1, m_2}(X_1 \times X_2)$ be a classical bisingular pseudodifferential operator, then 
\begin{align}
\Wres^2(A) &= (2\pi)^{-n_1 - n_2} \int_{S^{\ast} X_1 \times S^{\ast} X_2} \sigma^{-n_1, -n_2}(A)(x_1, \xi_1, x_2, \xi_2) \,d\omega_1(x_1, \xi_1) \,d\omega_2(x_2, \xi_2). \label{Wres}
\end{align}

In particular the definition of $\Wres$ does not depend on the choice of the operators $Q_1$ and $Q_2$. 
\label{Thm:Trace2}
\end{Thm}


The following Theorem and its proof is adapted to the bisingular context from a result of M. Wodzicki, see \cite{kassel}, Prop. 1.3.  

\begin{Thm}
The bisingular Wodzicki residue $\Wres^{2}$ is a trace on $\Psi_{cl}^{\Zz,\Zz}$, i.e. $A \in \Psi_{cl}^{m_1, m_2}(X_1 \times X_2), \ B \in \Psi_{cl}^{p_1, p_2}(X_1 \times X_2)$ 
then we have $\Wres^2([A, B]) = 0$
\label{Thm:trace}
\end{Thm}

\begin{proof}
First rewrite $\Wres^2$ in terms of the spectral $\zeta$-function (cf. \cite{seeley}) as follows
\[
\Wres^2(A) = m_1' m_2' \frac{d}{du} (\Res_{z = 1}^2 \zeta(P + uA, z))_{|u = 0}.
\]

for an elliptic operator $P \in \Psi_{cl}^{m_1', m_2'}(X_1 \times X_2)$ and $z \in (-1,1)$.
Under the assumption that $A$ is elliptic and invertible we obtain
\begin{align*}
\Tr((P + uAB)^{-z}) &= \Tr(A^{-1} (P + uAB)^{-z} A) \\
&= \Tr((A^{-1} P A + u BA)^{-z}). 
\end{align*}

In general we can set $A(z,z) = A + z(I + Q_1)^{\frac{m_1}{2}} \otimes z(I + Q_2)^{\frac{m_2}{2}}$ for positive, elliptic
operators $Q_i \in \Psi_{cl}^{1}(X_i), \ i= 1,2$. 
Then $A(z,z)$ is clearly invertible for $|z|$ large. 
And hence it follows by the above calculation that
\[
\Wres(A(z,z) B) = \Wres(B A(z,z)), \ |z| \gg 0. 
\]

This equality still holds for $z = 0$ since both sides are of degree $1$ in $z$. 
\end{proof}


\begin{Def}
Given $A \in \Psi_{cl}^{m_1, m_2}(X_1 \times X_2)$ elliptic and selfadjoint the $\eta$-function is given by
\[
\eta(A, z) := \Tr A |A|^{-(z + 1)}, z \in \Cc. 
\]

\label{Def:eta}
\end{Def}


Using the trace property of the bisingular Wodzicki residue, \ref{Thm:trace} we obtain the following result (see also \cite{ponge}, Prop. 2.4).
\begin{Lem}
For a given selfadjoint, elliptic bisingular operator $A$ the $\eta$-function $\eta(A, z)$ is holomorphic outside $\Poles$
and on $\Poles$ has at worst order two pole singularities such that 
\begin{align}
& \Res_{z=\sigma}^2 \eta(A, z) = m_1 m_2 \Wres^{2}(F |A|^{-\sigma}), \ \sigma \in \Poles. \label{res2}
\end{align}

Here $F := A |A|^{-1}$ is the \emph{sign-operator}. 

\label{Lem:Wres}
\end{Lem}

\begin{proof}
This follows from Theorem \ref{Thm:TRb}. 
\end{proof}

If $A \colon C^{\infty}(X_1 \times X_2) \to C^{\infty}(X_1 \times X_2)$ denotes a classical bisingular pseudodifferential operator on smooth compact manifolds $X_1, X_2$ which is self-adjoint of orders $m_1 > 0, m_2 > 0$
we designate by the indices $\downarrow$ the spectral cut in the lower halfplane $\Im \lambda < 0$ and by $\uparrow$ the spectral cut in the upper halfplane $\Im \lambda > 0$. 

We fix the notation $\Pi_{\pm}(A)$ for the projection onto the positive respectively negative eigenspace of $A$.

\begin{Thm}
Let $A$ be a self-adjoint, classical bisingular operator with assumption \ref{Ass:A}. 
Then for $k \in \Nn$ we have the identity
\[
\Res_{z=0}^2 \eta(A, z) = \Res_{z=0}^k (\zeta_{\down}(A, z) - \zeta_{\up}(A, z)) - \Res_{z=0}^{k+1} \zeta_{\up}(A, z).
\]

In particular for $k = 2$ we obtain
\begin{align}
& \Res_{z=0}^2 \eta(A, z) = m_1 m_2 \Wres^2(A |A|^{-1}) = \Res_{k=0}^2 (\zeta_{\down}(A, z) - \zeta_{\up}(A, z)). \label{wres}
\end{align}

\label{Thm:identity}
\end{Thm}

\begin{proof}
The proof is adapted to our case from \cite{ponge}. 
For $A = A^{\ast}$ selfadjoint we set $F = A |A|^{-1}$ for the sign-operator of $A$. 
This can also be written
\[
F = \Pi_{+}(A) - \Pi_{-}(A).
\]

Then from the definition of $\eta$ it follows that we can write
\[
\eta(A, z) = \Tr F |A|^{-z}, z \in \Cc. 
\]

%

By definition the difference $\zeta_{\up}(A, z) - \zeta_{\down}(A, z)$ to $\eta(A, z)$ is related by the sign operator
\begin{align}
A_{\up}^{z} &= \Pi_{+}(A) |A|^z + e^{-i\pi z} \Pi_{-}(A) |A|^z, \label{one} \\
A_{\down}^z &= \Pi_{+}(A) |A|^z + e^{i\pi z} \Pi_{-}(A) |A|^z, \label{two}. 
\end{align}

Thus 
\begin{align}
A_{\up}^z - A_{\down}^z &= (e^{-i \pi z} - e^{i\pi z}) \Pi_{-}(A) |A|^z = (1 - e^{2 \pi i z}) \Pi_{-}(A) A_{\up}^z. \label{three}
\end{align}

We get from \eqref{one} and \eqref{two}:
\[
A_{\up}^z - F |A|^z = (1 + e^{-\pi i z}) \Pi_{-}(A) |A|^z. 
\]

Use \eqref{three} and $(1 - e^{i \pi z})(e^{-i \pi z} + 1) = e^{-i \pi z} - e^{i \pi z}$ to obtain
\begin{align}
A_{\up}^z - A_{\down}^z &= (e^{-i \pi z} - e^{i \pi z})(1 + e^{i \pi z})^{-1} (A_{\up}^z - F |A|^z) = (1 - e^{i \pi z})(A_{\up}^z - F |A|^z). \label{fourth}
\end{align}

Via $\eta(A, z) = \Tr F|A|^{-z}$ it follows with this
\[
\zeta_{\up}(A, z) - \zeta_{\down}(A, z) = (1 - e^{-i \pi z}) \zeta_{\up}(A, z) - (1 - e^{-i \pi z}) \eta(A, z), z \in \Cc. 
\]

Equivalently, write
\begin{align}
& \eta(A, z) = \zeta_{\down}(A, z) - \zeta_{\up}(A, z) + (1 - e^{i \pi z}) \zeta_{\up}(A, z). \label{P}
\end{align}

Apply the $k$-th residue to both sides of \eqref{P}
\[
\Res_{z=0}^k \eta(A, z) = \Res_{z=0}^k (\zeta_{\down}(A, z) - \zeta_{\up}(A, z)) - \Res_{z=0}^{k+1} \zeta_{\up}(A, z)
\]

By Lemma \ref{Lem:Wres} this yields for $k = 2$
\[
m_1 m_2 \Wres(F) = \Res_{z=0}^2 (\zeta_{\down}(A, z) - \zeta_{\up}(A, z)) - \Res_{z=0}^{3} \zeta_{\up}(A, z). 
\]

For $k = 2$ we therefore obtain the result since the residue on the right hand side is $0$. 
Consider the case $k = 1$ and rewrite the residue on the right hand side as follows (inserting the definition of $\zeta_{\up}$)
\begin{align*}
\Res_{z=0}^2 \zeta_{\up}(A, z) &= \Res_{z=0}^2 \Tr((\Pi_{+}(A) + e^{-i \pi z} \Pi_{-}(A) |A|^z) \\
&= \Res_{z=0}^2 \Tr((\Pi_{+}(A) + \Pi_{-}(A) |A|^z)) + \Res_{z=0}^2 (e^{-i \pi z} - 1) \Tr \Pi_{-}(A) |A|^z 
\end{align*}

Note that the second expression with the trace on the right hand side has a second order pole. 
The residue of the second expression on the right hand side is therefore $0$ leaving the term 
\[
\Res_{z=0}^2 \zeta_{\up}(A, z) = \Res_{z=0}^2 \Tr |A|^z = 0. 
\]

Hence we obtain that $\Res_{z=0} \eta(A, z) = \Res_{z=0} (\zeta_{\up}(A, z) - \zeta_{\down}(A, z))$ as requested. 
\end{proof}


\section{Regularity of the $\eta$-invariant}

In order to derive the main result we need to establish algebraic topological properties of the bisingular class.
For this we state the following result (cf. also \cite{bs} or \cite{bohlen}):
\emph{A bisingular operator $A \in G_{cl}^{m_1, m_2}(\Rr^{n_1 + n_2}, \L(\Cc^{k \times l}))$ is elliptic if and only if $A$ is a Fredholm operator $A \colon Q^{s, t}(\Rr^{n_1 + n_2}; \Cc^k) \to Q^{s - m_1, t - m_2}(\Rr^{n_1 + n_2}; \Cc^l)$ for some / all $(s,t) \in \Rr^2$.}

\begin{Lem}
\emph{i)} The classical bisingular operators $G_{cl}^{0,0}(\Rr^{n_1} \times \Rr^{n_2}, \L(\Cc^{k \times l}))$ form a $\Psi^{\ast}$-algebra.

\emph{ii)} The algebra $G_{cl}^{0,0}(\Rr^{n_1} \times \Rr^{n_2}, \L(\Cc^{k \times l}))$ is closed under holomorphic functional calculus. 
\label{Lem:psistar}
\end{Lem}

\begin{proof}
\emph{i)} Let $A \colon L^2(\Rr^{n_1 + n_2}; \Cc^{k \times l}) \iso L^2(\Rr^{n_1 + n_2}; \Cc^{k \times l})$ be a linear isomorphism. 
We need to prove that $A$ is in particular invertible in $G_{cl}^{0,0}(\Rr^{n_1 + n_2}; \L(\Cc^{k \times l}))$. 
To this end note that by the above result we have $A = \op(a)$ for an elliptic symbol $a \in \Gamma_{cl}^{0,0}$. 
The operator $A$ is Fredholm by assumption, hence there is a $b \in \Gamma_{cl}^{0,0}$ such that
\[
\op(a) \op(b) = I + R_1, \ \op(b) \op(a) = I + R_2
\]

for $R_i$ a smoothing operator. Then observe that
\[
A^{-1} = \op(b) - \op(a) R_2 + R_1 A^{-1} R_2
\]

where the inverse on the right-hand side denotes the inverse in the bounded operators on $L^2$. 
Hence we have shown that
\[
\inv(\L(L^2)) \cap \A = \inv(\A)
\]

which verifies the $\Psi^{\ast}$-property. 

\emph{ii)} Let $A \in G_{cl}^{0,0}(\Rr^{n_1} \times \Rr^{n_2}, \L(\Cc^{k \times l}))$ and denote by $\sigma(A)$ the spectrum of $A$ considered as an operator in $\L(L^2(\Rr^{n_1} \times \Rr^{n_2})$. 
Denote by $\gamma$ a curve around $\sigma(A)$ within an open set $\Omega \subset \Cc$. 
By the above result we know that $A$ is invertible if $A \in \inv \L(L^2)$. 
Since $G_{cl}^{0,0}$ is a Fr\'echet algebra the group $\inv(G_{cl}^{0,0})$ is also open. Therefore by a classical result the inversion is continuous.
Hence the operator
\[
f(A) = \frac{1}{2\pi i} \int_{\gamma} f(z) (z - A)^{-1} \,d\lambda
\]

exists and is contained in $G_{cl}^{0,0}$. Therefore $G_{cl}^{0,0}$ is closed under holomorphic functional calculus. 
\end{proof}

\begin{Prop}
Let $A \in G_{cl}^{m_1, m_2}(\Rr^{n_1 + n_2}, \ \L(\Cc^{k \times l}))$ be a self-adjoint, elliptic operator with $m_1, \ m_2 > 0$. 
Assume that $\Lambda$ is contained either in the upper half-plane $\{z \in \Cc : \Im(z) > 0\}$ or the lower half-plane
$\{z : \in \Cc : \Im(z) < 0\}$. Then the symbol $A$ fufills the assumption \ref{Ass:A}. 
\label{Prop:sa}
\end{Prop}

\begin{proof}
\emph{1)} Consider first the joint principal symbol. We have $\sigma^{m_1, m_2}(A^{\ast}) = \sigma^{m_1, m_2}(A) = \sigma^{m_1, m_2}(A)^{\ast}$, hence
the spectrum is contained in $\Rr$. This implies that the first condition needed for $\Lambda$-ellipticity holds. 
Next it holds that $\sigma_1(A^{\ast}) = \sigma_1(A)^{\ast} = \sigma_1(A)$, hence the principal symbol is self-adjoint valued.
The same holds for the second principal symbol. Therefore the second and third condition hold as well. 
Hence we have verified part \emph{i)} of the assumption \ref{Ass:A}. 

\emph{2)} From the self-adjointness of $A$ it follows that the spectrum $\sigma(A)$ is contained in $\Rr$. 
Hence $\sigma(A) \cap \Lambda = \emptyset$ which verifies part \emph{ii)} of the assumption \ref{Ass:A}. 
\end{proof}

\begin{Thm}
The $\eta$-invariant $\eta(A, z)$ for a given positive, elliptic and self-adjoint classical bisingular operator $A$ has at most first order poles in $z = 0$. 
\label{Thm:eta}
\end{Thm}

\begin{proof}
We consider the case $\frac{n_1}{m_1} = \frac{n_2}{m_2}$ for in this case the $\eta$ function can have poles of second order.
Setting $\A := \overline{G_{cl}^{0,0}(\Rr^{n_1 + n_2})}^{\L(L^2)}$ we know from section \ref{Kthy} that
$\A \cong A_1 \otimes A_2$. Since the $K$-groups of $A_1$ and $A_2$ are torsion free we can apply K\"unneth's theorem
to obtain the $K$-theory $K_0(\A) \cong \Zz \otimes \Zz$ with generator $[I_1]_0 \otimes [I_2]_0 = [I_1 \otimes I_2]_0$. 

By spectral invariance the inclusion $j \colon G_{cl}^{0,0} \hookrightarrow \A$ induces an isomorphism in $K$-theory
\[
K_0(j) \colon K_0(G_{cl}^{0,0}) \iso K_0(\A). 
\]

Fix a trace $\tau^{2} \colon \A \to \Cc$ on $\A$.  
Let $p \in G_{cl}^{0,0}$ be an idempotent, then there is a unique homomorphism in $K$-theory
$K_0(\tau^{2}) \colon K_0(\A) \to \Cc$ making the following diagram commute
\[
\xymatrix{ 
K_0(G_{cl}^{0,0}) \ar@{>->>}[r]^{K_0(j)} & K_0(\A) \ar[r(1.4)]^-{K_0(\tau^2)} & & \Cc \\
\mathrm{Proj}(G_{cl}^{0,0}) \ar[u]^{[\cdot]_0}  \ar[urrr]^{\tau^2} & & &
}
\]

such that
\[
K_0(\tau^{2})([p]_0) = \tau^{2}(p). 
\]

In particular we see that
\[
K_0(\Wres^{2})([I_1]_0 \otimes [I_2]_0) = \Wres^{2}(I_1 \otimes I_2).
\]

The latter trace is zero which can be seen by expressing the trace in terms of the asymptotic expansion which only
depends on the $a_{-n_1, -n_2}$ term by an application of Theorem \ref{Thm:Trace2} which carries over for the global bisingular operators.
Hence $K_0(\Wres^{2}) = 0$ and thus $\Wres^{2}(p) = 0$ for any idempotent $p$ in $G_{cl}^{0,0}$. 
Since the double residue of the $\eta$ function is expressed as the Wodzicki residue of an idempotent (cf. Lemma \ref{Lem:Wres}) we
obtain that the double residue vanishes. Therefore $\eta$ does not have poles of second order in $z = 0$.
\end{proof}


\section{Concluding remarks}


The previous arguments can be easily adapted to the SG-calculus. 
In particular it is known that the classical SG-class is closed under complex powers (cf. \cite{bc}, \cite{mss}). 
Additionally, we can define the $\eta$-function and adapt the proof of this paper to the SG case.
The $K$-theory of completed SG-operators has been calculated in \cite{nicola}. 

We have in this paper restricted discussion of the regularity of the $\eta$-function to the \emph{global} calculus.
It is a future goal to also study the regularity properties for the bisingular calculus on closed manifolds.
In this case more advanced techniques are needed.


\end{document}